\documentclass{amsart}

\usepackage{amsmath,amsthm,amssymb,mathtools}
\usepackage{url}
\usepackage{tikz-cd}

\theoremstyle{plain}
\newtheorem{theorem}{Theorem}
\newtheorem{lemma}[theorem]{Lemma}
\newtheorem{prop}[theorem]{Proposition}
\newtheorem{corollary}[theorem]{Corollary}
\newtheorem{definition}[theorem]{Definition}
\newtheorem{alg}[theorem]{Algorithm}
\theoremstyle{remark}
\newtheorem{example}[theorem]{Example}
\newtheorem{remark}[theorem]{Remark}
\numberwithin{theorem}{section}
\numberwithin{equation}{section}

\newcommand{\Q}{\mathbb{Q}}

\newcommand{\Z}{\mathbb{Z}}
\newcommand{\F}{\mathbb{F}}
\newcommand{\PP}{\mathbb{P}}
\newcommand{\fraka}{\mathfrak{a}}

\newcommand{\frakp}{\mathfrak{p}}
\newcommand{\frakq}{\mathfrak{q}}
\newcommand{\OO}{\mathcal{O}}
\newcommand{\OK}{{\tilde{\mathcal{O}}}}
\newcommand{\frakA}{\mathfrak{A}}

\newcommand{\frakP}{\mathfrak{P}}
\newcommand{\frakQ}{\mathfrak{Q}}
\newcommand{\calI}{\mathcal{I}}
\newcommand{\calP}{\mathcal{P}}
\newcommand{\calF}{\mathcal{F}}
\newcommand{\isom}{\cong}

\DeclareMathOperator{\Div}{Div}
\DeclareMathOperator{\Pic}{Pic}
\DeclareMathOperator{\ddiv}{div}
\DeclareMathOperator{\ord}{ord}
\DeclareMathOperator{\Chow}{Chow}
\DeclareMathOperator{\len}{len}

\DeclareMathOperator{\Cl}{Pic}
\DeclareMathOperator{\im}{im}
\DeclareMathOperator{\Spec}{Spec}
\DeclareMathOperator{\coker}{coker}

\begin{document}
\bibliographystyle{plain}
\title{Chow groups of one-dimensional noetherian domains}

\author{Markus Kirschmer}
\address{Universit\"at Bielefeld, Postfach 100131, 33501 Bielefeld, Germany}
\email{markus.kirschmer@math.uni-bielefeld.de}
\author{J\"urgen Kl\"uners}
\address{Universit\"at Paderborn, Fakult\"at EIM, Institut für Mathematik, Warburger Str. 100, 33098 Paderborn, Germany}
\email{klueners@math.uni-paderborn.de}


\begin{abstract}
We discuss various connections between ideal classes, divisors, Picard and Chow groups of one-dimensional noetherian domains.
As a result of these, we give a method to compute Chow groups of orders in global fields and show that there are infinitely many number fields which contain orders with trivial Chow groups.
\end{abstract}

\maketitle

\section{Introduction}

Let $\OO$ be a one-dimensional noetherian domain with field of fractions $K$.
We assume that the normalization $\OK$ of $\OO$ is finitely generated over $\OO$. An important invariant of $\OO$ is the Picard group which is the quotient 
\[ \Cl(\OO):= \calI(\OO) / \{a \OO \mid a \in K^* \}\]
of all invertible fractional ideals modulo the principal fractional ideals of $\OO$. In the case $\OO=\OK$ this group is called the class group, and
every non-zero fractional ideal of $\OK$ is invertible and can be uniquely written as a product of maximal ideals.
A fundamental result of algebraic number theory states that the group $\Cl(\OK)$ is always finite if $K$ is a global field. 
The class number of $K$ is the order of the class group $\Cl(\OK)$.
We get that $\OK$ is a principal ideal domain if and only if it has class number 1.
The behavior of class groups is still mysterious and a famous open problem is the so-called class number 1 problem which is the question, if there exist infinitely many number fields with class number 1.
On the other hand we know many infinite families of number fields which have non-trivial class numbers, e.g. there are only finitely many imaginary quadratic number fields which have class number 1.

The Picard group gives no information about the structure of non-invertible ideals. The non-invertible prime ideals of $\OO$ are exactly the (finitely many) prime ideals of $\OO$ which contain the conductor
\( \calF := \{ x \in K \mid x\OK \subseteq \OO \} \). This
is the largest ideal of $\OK$ that is contained in $\OO$.  It is very classical (see Definition \ref{def:chow})  to define the Chow group 
\[ \Chow(\OO):= \Div(\OO) / \calP(\OO) \]
as the group of divisors $\Div(\OO)$ of $\OO$ modulo the principal divisors $\calP(\OO)$. 

The goal of this paper is to understand the behavior of Chow groups and its relation to $\Pic(\OO)$ or $\Cl(\OK)$.
For example, there is a canonical homomorphism from $\Pic(\OO) \rightarrow \Chow(\OO)$ and Proposition \ref{PicChow} discusses when this homomorphism is injective or surjective.
We are also interested in the relation between the Chow groups of two orders $\OO \subseteq \OO'$.
There is a canonical homomorphism $\Chow(\OO') \rightarrow \Chow(\OO)$ and its image and kernel are given by Lemma~\ref{lem:O1O2}.

Let $\frakp_1,\ldots,\frakp_r$ be the non-invertible maximal ideals of $\OO$. In Theorem \ref{Thm:GlobalChow} we show that we have an exact sequence
\[0 \rightarrow \Chow(\OK)/R \rightarrow \Chow(\OO) \rightarrow \bigoplus_{i=1}^r \Chow(\OO_{\frakp_i}) \rightarrow 0,\]
where $\Chow(\OO_{\frakp_i})$ is the Chow group of the localization of $\OO$ at $\frakp_i$ and $R$ denotes the kernel of the canonical homomorphism $\Chow(\OK) \to \Chow(\OO)$.
The local Chow groups are cyclic groups and the order can be computed by considering the degree of the residue field $\OO/\frakp_i$ and the degrees of the residue fields of prime ideals of $\OK$ lying above~$\frakp_i$.
In Theorem \ref{Thm:GlobalChow} we also show that the group of relations $R$ can be computed by considering all prime ideals of $\OK$ which lie over one of the prime ideals $\frakp_i$.
We remark that the above sequence is not necessarily split, see Example~\ref{ex:nonsplit}.

In the number field case, \cite{Picard} gives an algorithm to compute $\Pic(\OO)$ from $\Pic(\OK)$.
In the same vein, we show in Section 2 how to compute $\Chow(\OO)$ from $\Chow(\OK) \isom \Pic(\OK)$ provided we can factor the conductor of $\OO$.
It is well known how to do this in global fields.

Note that for a global function field $K/\F_q(t)$, we can consider the Chow group $\Chow(K)$ of the corresponding projective smooth curve. Following \cite[ch.~14]{Rosen} denote by $S_K$ all primes of $K$  and by $S$ the finitely many primes of $K$ above the infinite prime of $\F_q(t)$. Then the $S$-order $\OO_S:=\{a \in K \mid \ord_{\frakP}(a) \geq 0, \;\forall \frakP \in S_K\setminus S\}$ is our maximal order $\tilde\OO$. Then  $\Chow(K)$  is isomorphic to $\Z + \Chow^0(K)$, where $\Chow^0(K)=\Div^0(K)/\calP(K)$ are the divisors of degree 0 modulo the principal divisors. From \cite[Prop. 14.1]{Rosen} we get the exact sequence
\[0\rightarrow \Div^0(S)/\calP(S) \rightarrow \Chow^0(K) \rightarrow \Chow(\tilde\OO) \rightarrow C \rightarrow 0,\]
where $C$ is a cyclic group of order $d$, and $d$ is the gcd of the degrees of the divisors in $S$. The group $\Div^0(S)/\calP(S)$ is generated by the degree 0 divisors generated by the primes in $S$. This exact sequence is used in \cite[Kor 1.3]{Hess} to compute $\Chow(\tilde\OO)$ and the corresponding function SClassGroupExactSequence is available in Magma. 

In Theorem \ref{chow1} we show that there exist infinitely many number fields which admit an order with trivial Chow group.
As mentioned before, it is not known if there are infinitely many number fields with class number 1.

The isomorphism classes of fractional ideals of $\OO$ form a monoid, called the ideal class monoid.
Marseglia \cite{Marseglia} recently gave an algorithm to compute it in the case that $K$ is a number field.
Since the algorithm works out all over-orders of $\OO$, it can be quite time consuming.
The Chow group on the other hand is very easy to compute but yields less information than the ideal class monoid.
 

The paper is organized as follows.
In Section 2, we recall basic definitions regarding divisors and Chow groups.
We discuss canonical homomorphisms between such groups belonging to different rings.
In Theorem \ref{main} we show how to compute Chow groups and in Algorithm \ref{AlgPrinTest} we give a principal divisor test.
We end this section by analyzing the homomorphism between the Picard and Chow group of $\OO$.
Section 3 shows that many number fields contain orders with trivial Chow group and Section 4 gives some examples.
The last section corrects some wrong statements from \cite{MaximalI}. In Theorem \ref{prop:Fix} we describe all  $\OO$ such that $\OO^*=\OK^*$ and $\Pic(\OO) \isom \Cl(\OK)$.
In Theorem \ref{Thm:divinj} we show that the canonical homomorphism $\calI(\OO) \to \Div(\OO)$ is injective if and only if $\OO=\OK$.


\subsection*{Acknowledgments}
The authors thank Stefano Marseglia for helpful comments.
We are also grateful to the reviewer for numerous suggestions and the proof of part (2) of Lemma~\ref{Order_aOO}.

\section{Divisors and Chow groups}\label{Sec:Chow}

Let $\OO$ be a one-dimensional noetherian domain with field of fractions $K$.
Let $\OK$ be its normalization, i.e. the integral closure of $\OO$ in $K$.
We assume that $\OK$ is finitely generated over $\OO$.

One of the most prominent examples of such a setup is an algebraic number field $K$ with ring of integers $\OK$.
Then $\OO$ is an order in $K$, that is a subring of $K$ which is finitely generated as a $\Z$-module and that contains a $\Q$-basis of $K$.


Let $\PP(\OO)$ be the set of maximal ideals of $\OO$ and $\Spec(\OO) = \PP(\OO) \cup \{ (0) \}$ be the set of prime ideals of $\OO$. 
Suppose now $\OO' \subseteq \OK$ is a subring with $\OO \subseteq \OO'$.
For $\frakP \in \PP(\OO')$ the ideal $\frakp:= \frakP \cap \OO$ is maximal in $\OO$. This gives a map of affine schemes
\[ f \colon \Spec(\OO') \to \Spec(\OO),\; \frakP \mapsto \frakP \cap \OO \:. \]
By \cite[Lemma I.12.7]{Neukirch}, the degree $\deg_\frakp(\frakP)$ of $\OO'/\frakP$ over $\OO / \frakp$ is finite.

For a maximal ideal $\frakp \in \PP(\OO)$, let $\OO_\frakp$ be the localization of $\OO$ at $\frakp$.
More generally, if $M$ is an $\OO$-module, we denote by $M_\frakp:= M \otimes_\OO \OO_\frakp$ the localization of $M$ at $\frakp$.
We follow Neukirch \cite[Section I.12]{Neukirch} to define the divisor and Chow groups of $\OO$.
For $\frakp \in \PP(\OO)$ define
\[ \ord_\frakp \colon K^* \to \Z,\; \tfrac{a}{b} \mapsto \len_{\OO_\frakp}( \OO_\frakp / a \OO_\frakp) - \len_{\OO_\frakp}( \OO_\frakp / b \OO_\frakp) \]
where $\len_{\OO_\frakp}(M)$ denotes the composition length of an $\OO_\frakp$-module $M$.
Then $\ord_\frakp$ is a well defined homomorphism.

%

\begin{definition}\label{def:chow}
The group of divisors $\Div(\OO)$ is the free abelian group with the set of maximal ideals of $\OO$ as a basis.
The image $\calP(\OO)$ of the homomorphism
\[ \ddiv_\OO \colon K^* \to \Div(\OO),\; a \mapsto \sum_{\frakp \in \PP(\OO)} \ord_\frakp(a) \frakp \]
are the principal divisors of $\OO$ and 
\[ \Chow(\OO):= \Div(\OO) / \calP(\OO) \]
is the Chow group of $\OO$.
\end{definition}

More generally, the Chow group of a one-dimensional noetherian scheme is defined by $0$-cycles modulo rational equivalence to $0$, see \cite[Section 1.3]{Fulton}.
The $0$-cycles correspond to divisors in $\Div(\OO)$ and this correspondence identifies principal divisors with $0$-cycles which are rationally equivalent to $0$.

The set of invertible fractional ideals of $\OO$ forms an abelian group $\calI(\OO)$ with the set of principal ideals as a subgroup.
The quotient
\[ \Pic(\OO):= \calI(\OO) / \{a \OO \mid a \in K^* \}\]
is called the Picard group of $\OO$.
Being invertible is a local property.
For an invertible fractional ideal $\fraka$ of $\OO$ and $\frakp \in \PP(\OO)$ there exists some $a_\frakp \in K^*$ such that $\fraka_\frakp = a \OO_\frakp$, see \cite[Satz (12.4)]{Neukirch} for details.
Hence we can define
\[ \ddiv_\OO(\fraka):= \sum_\frakp \ord_\frakp(a_\frakp) \frakp \:, \]
which induces a homomorphism
\[ \ddiv_\OO \colon \calI(\OO) \to \Div(\OO)\:. \] 
It takes principal ideals to principal divisors, so it induces a homomorphism
\[ \overline{\ddiv}_\OO \colon \Pic(\OO) \to \Chow(\OO) \:. \]
Note that $\ddiv_\OO \colon \calI(\OO) \to \Div(\OO)$ extends the map $\ddiv_\OO \colon K^* \to \Div(\OO)$ from Definition~\ref{def:chow} in the sense that $\ddiv_\OO(a) = \ddiv_\OO(a \OO)$ for all $a \in K^*$.

The groups $\calI(\OK)$ and $\Div(\OK)$ are freely generated by the maximal ideals of $\OK$ hence $\ddiv_\OK\colon \calI(\OK) \to \Div(\OK)$ and $\overline{\ddiv}_\OK\colon \Pic(\OK) \to \Chow(\OK)$ are canonical isomorphisms.


We now come to a description of $\Chow(\OO)$, which is based on the following result.

\begin{prop}\label{Chow21}
\begin{enumerate}
\item
Let $a \in K^*$ and $\frakp \in \PP(\OO)$. If $\frakP_1,\dots,\frakP_r$ denote the prime ideals of $\OO'$ over $\frakp$, then
\[ \ord_\frakp(a) = \sum_{i=1}^r \ord_{\frakP_i}(a) \deg_{\frakp}(\frakP_i) \:. \]
\item
The map $f \colon \Spec(\OO') \to \Spec(\OO)$ induces the push-forward homomorphism
\[ f_* \colon \Div(\OO') \to \Div(\OO),\; \sum_{\frakP \in \PP(\OO')} a_\frakP \frakP \mapsto \sum_{\frakP \in \PP(\OO')} a_\frakP \deg_{f(\frakP)}(\frakP) \cdot f(\frakP) \:. \]
If $a \in K^*$, then
\begin{equation}\label{eq:fstar}
f_*(\ddiv_\OO'(a)) = f_* \left( \sum_{\frakP \in \PP(\OO')} \ord_{\frakP}(a) \frakP \right) = \sum_{\frakp \in \PP(\OO)} \ord_{\frakp}(a) \frakp  = \ddiv_\OO(a) \:.
\end{equation}
Hence $f_*$ maps principal divisors to principal divisors and therefore induces a homomorphism
\[ \overline{f_*} \colon \Chow(\OO') \to \Chow(\OO) \:. \]
\end{enumerate}
\end{prop}
\begin{proof}
See \cite[Example A.3.1]{Fulton} for a proof of part (1) and \cite[Theorem 1.4]{Fulton} or \cite[Lemma 42.18.1]{stacks-project} for a proof of (2).
\end{proof}

We will show in Example~\ref{ex:NeitherNor} that the homomorphism $\Chow(\OO') \to \Chow(\OO)$ is in general neither injective nor surjective.
Part (1) of Proposition~\ref{Chow21} immediately gives the following result.

\begin{corollary}\label{count}
Let $a \in \OK_\frakp$ be an element in the normalization.
Then $a$ is a unit if and only if $\ord_\frakp(a) = 0$.
\end{corollary}

\begin{example}\label{ex:root7}
Let $K = \Q(\sqrt{-7})$ and $\OO = \Z[\sqrt{-7}]$. Let $\overline{\phantom{x}} \colon K \to K$ denote the non-trivial automorphism.
Then $\OK = \Z[\alpha]$ where $\alpha:= \frac{1+\sqrt{-7}}{2}$.
The ideal $2 \OK$ factors into $\frakP \overline{\frakP}$ with the prime ideal $\frakP = \alpha \OK$ of norm $2$.
Consider the element
\[ a:= \frac{\alpha}{\overline{\alpha}} = \frac{1+\sqrt{-7}}{1-\sqrt{-7}} = \frac{-3 + \sqrt{-7}}{4} \in K^* \:. \]
Then $a\OK_\frakP = \frakP \OK_\frakP$. Hence $\ord_\frakP(a) = 1$ and similarly $\ord_{\overline{\frakP}}(a) = -1$.
Proposition~\ref{Chow21} shows that for $\frakp = \frakP \cap \OO = \overline{\frakP} \cap \OO$ we find $\ord_\frakp(a) = 1-1 = 0$.
\end{example}

The conductor 
\( \calF := \{ x \in K \mid x\OK \subseteq \OO \}\) of $\OO$ 
is the largest ideal of $\OK$ that is contained in $\OO$.
A maximal ideal $\frakp$ of $\OO$ is invertible (or regular) if and only if $\frakp$ is coprime to the conductor $\calF$.
If this is the case, then $\frakp \OK$ is the unique maximal ideal of $\OK$ over $\frakp = \frakp \OK \cap \OO$ and $\OK_{\frakp \OK} \isom \OO_\frakp$.
Thus $\OK/\frakp\OK \isom \OO/\frakp$ and therefore $\deg_\frakp(\frakp\OK) = 1$.
In particular, $f_*(\frakp\OO') = \frakp$ for any ring $\OO \subseteq \OO' \subseteq \OK$.
See \cite[Proposition I.12.10]{Neukirch} for details.
%
%

\begin{lemma}\label{lem:O1O2}
Let $\frakp_1,\dots,\frakp_r$ be the non-invertible maximal ideals of $\OO$.
For $1 \le i \le r$ let $\frakP_{i,1},\dots, \frakP_{i,r_i}$ be the maximal ideals of $\OO'$ over $\frakp_i$ and set
\[ g_i = \gcd(\deg_{\frakp_i}(\frakP_{i,1}), \dots, \deg_{\frakp_i}( \frakP_{i,r_i})) \:. \]
\begin{enumerate}
\item
The image and kernel of $f_* \colon \Div(\OO') \to \Div(\OO)$ are
\begin{align*}
 \im(f_*) &= \left\lbrace \sum_{\frakp \in \PP(\OO)} a_\frakp \frakp \mid a_{\frakp_i} \in g_i \Z \mbox{ for all } 1 \le i \le r \right\rbrace \\
\ker(f_*) &= \left\lbrace \sum_{i=1}^r \sum_{j=1}^{r_i} a_{i,j} \frakP_{i,j} \mid \sum_{j=1}^{r_i} a_{i,j} d_{i,j} = 0 \mbox{ for all } 1 \le i \le r  \right\rbrace \:.
\end{align*}
In particular, $\coker(f_*) \isom \prod_{i=1}^r \Z/g_i \Z$ and $f_*$ is injective if and only if $r_i=1$ for all $i$.
\item The image and kernel of $\overline{f_*} \colon \Chow(\OO') \to \Chow(\OO)$ are
\[ \im(\overline{f_*}) = \im(f_*)/\calP(\OO)  \quad \mbox{and} \quad  \ker(\overline{f_*}) = (\ker(f_*)+\calP(\OO'))/\calP(\OO') \:. \]
Moreover, $\coker(\overline{f_*}) \isom \coker(f_*)$.
\end{enumerate}
\end{lemma}
\begin{proof}
Part (1) follows from Proposition~\ref{Chow21} and the preceding discussion.
Equation~\eqref{eq:fstar} shows that $f_* \colon \calP(\OO') \to \calP(\OO)$ is surjective.
This gives the image and kernel of $\overline{f_*}$.
To prove the last assertion, consider the commutative diagram
\begin{center}
\begin{tikzcd}
0 \ar[r] & \calP(\OO) \ar[r] \ar[d] & \im(f_*) \ar[r] \ar[d] &\im(\overline{f_*}) \ar[r] \ar[d] & 0 \\
0 \ar[r] & \calP(\OO) \ar[r] & \Div(\OO) \ar[r] & \Chow(\OO) \ar[r] & 0
\end{tikzcd}
\end{center}
where the vertical arrows are the canonical inclusions. The snake lemma shows $\coker(\overline{f_*}) \isom \coker(f_*)$.
\end{proof}

For the remainder of the paper, we restrict to the case that $\OO' = \OK$ and fix the following notation. 
Let 
\[ f \colon \Spec(\OK) \to \Spec(\OO), \; \frakP \mapsto \frakP \cap \OO \]
and
let $\frakp_1,\dots,\frakp_r$ be the maximal ideals of $\OO$ that are not invertible.
For $1 \le i \le r$ let $\frakP_{i,1},\dots,\frakP_{i,r_i}$ be the maximal ideals of $\OK$ over $\frakp_i$.
Further, let $d_{i,j} = \deg_{\frakp_i}(\frakP_{i,j})$ be the degree of $\OK/\frakP_{i,j}$ over $\OO/\frakp_i$.
For $1 \le i \le r$ write
\begin{equation}\label{eq:gi}
 g_i := \gcd(d_{i,1},\dots, d_{i, r_i})  = \sum_{j=1}^{r_i} \lambda_{i,j} d_{i,j}
\end{equation}
with integers $\lambda_{i,j}$.
Then Lemma~\ref{lem:O1O2} yields the exact sequence
\[0 \rightarrow \im(\overline{f_*}) \rightarrow \Chow(\OO) \rightarrow \coker(\overline{f_*}) \isom \bigoplus_{i=1}^r \Z/g_i \Z \rightarrow 0 \:. \]


Since $\OK_{\frakp}$ has only finitely many prime ideals, it is a principal ideal domain.
Thus every ideal or divisor of $\OK_{\frakp}$ is principal and thus $\Chow(\OK_{\frakp}) \to \Chow(\OO_\frakp)$ is trivial.
Hence the preceding discussion yields the following result.

\begin{theorem}\label{Thm:GlobalChow}
\begin{enumerate}
\item The local Chow group $\Chow(\OO_\frakp)$ is trivial for every invertible prime ideal $\frakp$ of $\OO$.
\item The local Chow group $\Chow(\OO_{\frakp_i})$ is isomorphic to $ \Z/g_i \Z$ for all $1 \le i \le r$.
\item The sequence
\[0 \rightarrow \im(\overline{f_*}) \rightarrow \Chow(\OO) \rightarrow \bigoplus_{\frakp \in \PP(\OO)} \Chow(\OO_{\frakp}) \rightarrow 0  \]
is exact.
\end{enumerate}
\noindent{}Hence, $\Chow(\OO)$ is an extension of $\im(\overline{f_*}) \isom \Chow(\OK) / \ker(\overline{f_*})$
by the direct sum of the local Chow groups $\bigoplus_{i=1}^r \Chow(\OO_{\frakp_i}) \isom \bigoplus_{i=1}^r \Z/g_i \Z$.
\end{theorem}

The following theorem describes this extension, which is not necessarily split, c.f. Example~\ref{ex:nonsplit}.

\begin{theorem}\label{main}
For $1 \le i \le r$ let $Q_i = \sum_{j=1}^{r_i} \lambda_{i,j}  \frakP_{i,j}\in \Div(\OK)$, where the $\lambda_{i,j}$ are defined in Equation~\eqref{eq:gi}.
\begin{enumerate}
\item
The kernel of $f_*$ is generated by
\[ \{ \tfrac{d_{i,j}}{g_i}Q_i - \frakP_{i,j} \mid 1 \le i \le r,\; 1 \le j \le r_i \} \:. \]
\item
Let $G = \bigoplus_{i=1}^r \Z \frakp_i \oplus \Chow(\OK)/\ker(\overline{f_*})$. Then
\[ \varphi \colon G \to \Chow(\OO),\, (\sum_{i=1}^r a_i \frakp_i, [A]) \mapsto [ \sum_{i=1}^r a_i \frakp_i + f_*(A)) ] \]
is an epimorphism with kernel
\[ R = \langle \{ (g_i\frakp_i, [-Q_i]) \in G \mid 1 \le i \le r \} \rangle \:. \]
In  particular, $G/R \isom \Chow(\OO)$.
\end{enumerate}
\end{theorem}
\begin{proof}
(1) The given generators lie in $\ker(f_*)$.
Conversely, let $D := \sum_{i,j} a_{i,j} \frakP_{i,j} \in \ker(f_*)$.
Then $\sum_j a_{i,j} d_{i,j} = 0$ for all $i$, see Lemma~\ref{lem:O1O2}.
Therefore
\[ D + \sum_{i,j} a_{i,j}(\tfrac{d_{i,j}}{g_i} Q_i - \frakP_{i,j}) = \sum_{i,j} \tfrac{a_{i,j}d_{i,j}}{g_i} Q_i = \sum_i \tfrac{1}{g_i} (\sum_j a_{i,j} d_{i,j}) Q_i = 0 \:.\]
The result follows.

(2)
Lemma~\ref{lem:O1O2} shows that $\varphi$ is a well-defined epimorphism.
The subgroup $R$ is in its kernel since $f_*(Q_i) = g_i \frakp_i$.
Conversely, suppose that $( \sum_{i=1}^r a_i \frakp_i, [A])$ is in the kernel of $\varphi$.
Then $\sum_{i=1}^r a_i \frakp_i$ is in the image of $f_*$. By Lemma~\ref{lem:O1O2} this implies $g_i \mid a_i$ for all $i$.
So $(0, [A  + \sum_{i=1}^r \frac{a_i}{g_i}Q_i])$ is also in the kernel of $\varphi$.
By (1) this is only possible if $A  + \sum_{i=1}^r \frac{a_i}{g_i}Q_i \in \ker(f_*)$, i.e. $[A  + \sum_{i=1}^r \frac{a_i}{g_i}Q_i ] \in G/\ker(\overline{f_*})$ is trivial.
Hence $R$ is the kernel of $\varphi$.
\end{proof}

In the next step we give  an algorithm to detect if a given divisor is principal.

\begin{alg}\label{AlgPrinTest}
Given a divisor $D \in \Div(\OO)$ we can decide if $D$ is principal and if so, find some $\alpha \in K^*$ such that $\ddiv_\OO(\alpha) = D$.
\begin{enumerate}
\item If $D \notin \im(f_*)$ return false.
\item Using Proposition~\ref{Chow21}, find $A \in \Div(\OK)$ such that $f_*(A) = D$.
\item Compute $\Chow(\OK) \isom \Pic(\OK)$.
\item From the factorization of the conductor $\calF$ in $\OK$ obtain the non-invertible maximal ideals of $\OO$ and generators of $\ker(f_*)$, cf. Theorem~\ref{main}.
\item If $[A] \notin \ker(\overline{f_*})$, return false.
\item Pick some $B \in \ker(f_*)$ such that $A+B = \ddiv_\OO(\alpha)$ is principal.
\item Return $\alpha$.
\end{enumerate}
\end{alg}
\begin{proof}
Suppose first $D = \ddiv_\OO(\beta)$ is a principal divisor.
Equation~\eqref{eq:fstar} shows that $D = f_*(\ddiv_\OK(\beta)) \in \im(f_*)$.
The choice of $A$ yields $D = f_*(A)$ and thus $A - \ddiv_\OK(\beta) \in \ker(f_*)$.
Hence there exists some $B \in \ker(f_*)$ such that $A + B$ is principal.
So the element $\alpha$ from step (6) of the algorithm satisfies
\[ \ddiv_\OO(\alpha) = f_*(A) + f_*(B) = D + 0 = D \:. \]
Conversely, if the algorithm returns some element $\alpha \in K^*$, then the same argument as before shows that $D = \ddiv_\OO(\alpha)$ is principal.
So the output is always correct.
\end{proof}

\begin{remark}
Provided that we can compute $\Chow(\OK) \isom \Cl(\OK)$ and the conductor $\calF$ as well as factor $\calF$ into prime ideals, Theorem~\ref{main} yields an effective method to compute $\Chow(\OO)$.
To be able to run Algorithm~\ref{AlgPrinTest}, we also need to check if a given divisor of $\OK$ (or equivalently a given fractional ideal of $\OK$) is principal and if so, compute a generator.
For orders $\OO$ in a global field $K$, there are well known algorithms to solve these tasks.
\end{remark}


We close this section by discussing the canonical homomorphism $\Pic(\OO) \to \Chow(\OO)$.
The homomorphism $\calI(\OO) \to \calI(\OK), \fraka \mapsto \fraka \OK$ maps principal ideals to principal ideals and therefore it induces a canonical homomorphism
\[ \Pic(\OO) \to \Cl(\OK),\; [\fraka] \mapsto [ \fraka \OK] \:. \]
Since every element of  $\Cl(\OK)$ is represented by a product of prime ideals of $\OK$ coprime to $\calF$, the canonical homomorphism
$\Pic(\OO) \to \Cl(\OK)$
is surjective, see also \cite[p. 79]{Neukirch}.

Proposition~\ref{Chow21} shows that the homomorphism $\overline{\ddiv}_\OO \colon \Pic(\OO) \to \Chow(\OO)$ is the composition of the canonical homomorphisms
\[ \Pic(\OO) \to \Cl(\OK) \isom \Chow(\OK) \to \Chow(\OO) \:. \]

\begin{prop}\label{PicChow}
The homomorphism
\[ \overline{\ddiv}_\OO \colon \Pic(\OO) \to \Chow(\OO) \]
is surjective if and only if all local Chow groups of $\OO$ are trivial.
It is injective if and only if $\Pic(\OO) \to \Cl(\OK)$ is an isomorphism and the image of $\ker(\overline{f_*})$ is trivial.
\end{prop}
\begin{proof}
Since $\Pic(\OO) \to \Cl(\OK)$ is surjective, the map $\overline{\ddiv}_\OO$ is surjective if and only if $\Chow(\OK) \to \Chow(\OO)$ is.
By Lemma~\ref{lem:O1O2} that holds if and only if $\overline{f_*}$ is surjective.
Theorem~\ref{Thm:GlobalChow} shows that this is equivalent to all local Chow groups being trivial.\\
Similarly, $\overline{\ddiv}_\OO$ is injective if and only if $\Pic(\OO) \to \Cl(\OK)$ and $\Chow(\OK) \to \Chow(\OO)$ are.
So the second assertion follows again from Lemma~\ref{lem:O1O2}.
\end{proof}





\section{Orders in number fields with trivial Chow groups}

We are going to show that number fields often contain orders with trivial Chow group.
We start with a more general result which might be of independent interest,
but we will only use it in the case that $R=\Z$ and $\OK$ is the maximal order of an algebraic number field.
In this case, part (2) can also be deduced from Furtw\"angler's classification of conductors, see \cite{Furtwaengler,Grell,Lettl}.

\begin{lemma}\label{Order_aOO}
Let $R \subsetneq \OK$ be a finite extension of Dedekind rings.
Let $\fraka$ be a non-zero ideal of $R$ and set $\OO = R + \fraka \OK$.
\begin{enumerate}
\item $\OO$ is an $R$-subalgebra of $\OK$ which is finitely generated over $R$ and has the same field of fractions as $\OK$.
\item The conductor of $\OO$ is $\fraka \OK$.
\item Let $\frakp \in \PP(R)$ with $\fraka \subseteq \frakp$.
Then $\frakP:= \frakp + \fraka \OK$ is the unique maximal ideal of $\OO$ over $\frakp$.
It is non-invertible and $\OO / \frakP \isom R / \frakp$.
\end{enumerate}
\end{lemma}
\begin{proof}
Part (1) is clear.

(2) Let $\calF$ be the conductor of $\OO$. Then $\fraka \OK \subseteq \calF$.
To show that $\fraka \OK = \calF$, we may pass to the localizations of $R$ and assume that $R$ is local.
Then $\fraka = fR$ and $\calF \cap R = gR$ for some $f,g \in R$.
From $f \in \calF \cap R = gR$ we find $r \in R$ such that $f = rg$.
Let $x \in \calF$ and write $x = s + f y$ with $s \in R$ and $y \in \OK$.
Then $s \in R \cap \calF = gR$ and $x \in gR + f\OK \subseteq g \OK$. This shows $\calF = g\OK$.
The inclusions
\[ \OO = R + rg \OK = R + \fraka \OK \subseteq R + \calF  = R + g \OK \subseteq \OO  \]
show $R + rg \OK = R + g \OK$ and thus
\[ r \cdot \OO/R = r \cdot (R + g\OK)/R = (R + rg\OK )/R = \OO/R \:. \]
If $r$ was not a unit, Nakayama's lemma gives $\OO/R = 0$ and thus $R = \OO$.
This case is excluded. So $r \in R^*$ and thus $\calF = g\OK = f \OK = \fraka \OK$ as claimed.

(3) $\frakP$ is a proper ideal of $\OO$.
Thus the inclusions $\frakp \subseteq R \cap \frakP \subsetneq R$ imply $\frakp = \frakP \cap R$.
So the epimorphism $R \to \OO/\frakP$ has kernel $\frakp$ and therefore $\OO/\frakP \isom R / \frakp$.
Thus $\frakP$ is maximal and it is non-invertible as $\frakP + \calF = \frakp + \fraka \OK = \frakP$.
Let $\frakQ$ be a maximal ideal of $\OO$ over $\frakp$. Then
\[\frakP^2 = \frakp^2 + \frakp \fraka \OK + \fraka^2 \OK = \frakp^2 + \frakp \calF \subseteq \frakp \OO \subseteq \frakQ \]
implies that $\frakP \subseteq \frakQ$ since $\frakQ$ is prime. Hence $\frakP = \frakQ$.
\end{proof}

For the remainder of this section, let $K$ be an algebraic number field with ring of integers $\OK$.
We will use the previous result with $R=\Z$ to construct orders in $K$ with trivial Chow groups.

\begin{theorem}\label{ChowTrivial}
Suppose $K$ has a subfield $L$ of class number one.
If the relative index $m:= [K:L]$ is coprime to the class number $\# \Cl(\OK)$ then there exists an order in $K$ with trivial Chow group.
\end{theorem}
\begin{proof} We may assume that $K \ne \Q$.
By \cite[Theorem 1]{Primes} the class group  $\Cl(\OK)$ can be generated by a finite set $\mathcal{P}$ of maximal ideals of $\OK$ of (absolute) degree $1$.
Let $p_1,\dots,p_r$ be the different prime numbers below the ideals in $\mathcal{P}$.
By Lemma~\ref{Order_aOO}, the order $\OO = \Z + (p_1 \cdot \ldots \cdot p_r) \OK$ has conductor $(p_1 \cdot \ldots \cdot p_r) \OK$ and there exists a unique maximal ideal $\frakp_i$ of $\OO$ over $p_i$.
Theorem~\ref{Thm:GlobalChow} therefore implies that $\Chow(\OK_{\frakp_i})$ is trivial and thus $\overline{f_*} \colon \Chow(\OK) \to \Chow(\OO)$ is onto.\\
Let $\OK_L$ be the ring of integers of $L$ and let $\frakP$ be a maximal ideal of $\OK$ over $p_i$ of degree $1$.
Then $\frakP \cap \OK_L$ is principal, generated by $a \in L^*$ say.
So we have the following inclusions
\begin{center}
\begin{tikzcd}
\Z \ar[r, hook] & \OK_L \ar[r, hook] & \OK \\
p_i\Z \ar[r, hook] \ar[u, hook] & a\OK_L \ar[r, hook] \ar[u, hook] & \frakP \ar[u, hook]
\end{tikzcd}
\end{center}
Since $a \OK_L = \frakP \cap \OK_L$ has degree $1$, the integral element $a \in \OK$ has absolute norm $p_i^m$.
Hence $a \in K$ satisfies $\ord_{\frakp_i}(a) = m$ and $\ord_\frakq(a) = 0$ for all $\frakq \in \PP(\OO)$ different from $\frakp_i$.
Thus 
\[ f_*(m \frakP) = m\frakp_i = \ddiv_\OO(a) \overset{\eqref{eq:fstar}}{=} f_*(\ddiv_\OK(a)) \]
and 
$ [ \frakP]^m = [\frakP^{m} / a] $
lies in the kernel of $\Cl(\OK) \isom \Chow(\OK) \to \Chow(\OO)$. Since $m$ is coprime to the class number of $\OK$, we conclude that $[\frakP]$ itself lies in the kernel.
Since the maximal ideals of $\OK$ of degree 1 over the primes $p_1,\dots,p_r$ generate the class group of $\OK$, Theorem~\ref{Thm:GlobalChow} shows that $\Chow(\OO) = \im(\overline{f_*})$ is trivial.
\end{proof}

If we set $L = \Q$ in Theorem \ref{ChowTrivial} we get the following result.

\begin{corollary}\label{ChowTrivialQ}
Let $K$ be a number field of degree $n$.
If $n$ is coprime to the class number $\# \Cl(\OK)$ then $K$ contains an order $\OO$ with trivial Chow group.
\end{corollary}

The next example shows that we cannot remove the coprimality condition in Corollary~\ref{ChowTrivialQ}.

\begin{example} 
Let $\OO$ be an order in a quadratic number field $K$ with even class number.
We claim that the Chow group of $\OO$ cannot be trivial. 
If $\OO = \OK$ is maximal, then $\Chow(\OO) \isom \Cl(\OK)$ is not trivial.
Suppose now $ \OO \ne \OK$. Then $\OO = \Z + a \OK$ where $a$ denotes the index of $\OO$ in $\OK$.
If $a$ has a prime divisor $p$ which is inert in $K$, then $\Chow(\OO_p) \isom \Z/2\Z$ implies that $\Chow(\OO)$ is not trivial.
So we may assume that all the prime divisors $p_1,\dots,p_r$ of $a$ are ramified or split in $K$.
Then all local Chow groups of $\OO$ are trivial.
For $1 \le i \le r$ let $\frakP_i$ be a prime ideal of $\OK$ over $p_i$ and denote the non-trivial automorphism of $K$ by $\overline{\phantom{X}}$.
Theorem~\ref{Thm:GlobalChow} shows that
\begin{align*}
\Chow(\OO) &\isom 
\Cl(\OK) / \langle \{ [\frakP_i \overline{\frakP_i}^{-1}] \mid 1 \le i \le r,\; p_i \mbox{ splits in $K$} \} \rangle \\
&= \Cl(\OK) / \langle \{ [\frakP_i^2] \mid 1 \le i \le r,\; p_i \mbox{ splits in $K$} \} \rangle
\end{align*}
is not trivial.
\end{example}

\begin{example}
Let $p>2$ be prime and set $K = \Q(\sqrt{p^*})$ where $p^* = (-1)^{(p-1)/2} p$.
The discriminant of $K$ is $p^*$. Gau\ss{}' genus theory implies that $\Cl(\OK)$ is odd and Corollary~\ref{ChowTrivialQ} shows that $K$ contains an order $\OO$ with trivial Chow group.
\end{example}

The previous example shows the following result.

\begin{theorem}\label{chow1}
There are infinitely many (quadratic) number fields that contain orders with trivial Chow group.
\end{theorem}


The result of Theorem \ref{chow1} is a little bit surprising because for the ordinary class group it is not known that there are infinitely many number fields with class number 1. This is the famous class number 1 problem.
 
As soon as we find a number field, where the class number is coprime to the degree, we can find an order with trivial Chow group. The following example shows that we might be lucky in more cases.

\begin{example}
The maximal order $\OK$ of
\[ K = \Q[X]/(X^6 - X^5 + 25X^4 - 30X^3 + 603X^2 - 648X + 729) \] has class group isomorphic to $\Z/6\Z \times \Z/2\Z$.
The field $K$ has two non-trivial subfields $K_1 = \Q(\sqrt{-3})$ and $K_2 = \Q[X]/(X^3 - X^2 - 24X + 27)$; both with trivial class groups.
We claim that  $\OO:= \Z + 7 \OK$ has trivial Chow group.
This can be seen as follows. The prime $7$ splits completely in $\OK$ and thus $\overline{f_*}$ is onto.
Moreover, the prime ideals of $\OK$ over $7$ generate $\Cl(\OK)$.
Using $L=K_1$ ($L=K_2$) in the proof of Theorem~\ref{ChowTrivial}, we see that all squares (cubes) of $\Cl(\OK)$ lie in the kernel of $\Cl(\OK) \isom \Chow(\OK) \to \Chow(\OO)$.
But then  $\Chow(\OO) = \im(\overline{f_*})$ is trivial.
\end{example}

We close this section by showing that the number of orders in $K$ with trivial Chow group is either zero or infinite.

\begin{lemma}
If $K$ contains an order with trivial Chow group, then there exist infinitely many such orders in $K$.
\end{lemma}
\begin{proof}
Let $\OO'$ be an order in $K$ has trivial Chow group and set $m = [ \OK : \OO' ]$.
By Chebotarev's Density Theorem, there exist infinitely many primes $p$ coprime to $m$ which completely split in $K$.
Hence it suffices to show that $\OO := \OO' \cap (\Z + p \OK)$ has trivial Chow group.
Let $q$ be a prime divisor of $m$ and let $\frakq$ be a prime ideal of $\OO$ over $q$.
Localizing at $\frakq$ shows that there exists a unique prime ideal $\frakQ$ of $\OO'$ over $\frakq$ and $\OO/\frakq \isom \OO'/\frakQ$.
By Lemma~\ref{Order_aOO} the order $\Z + p \OK$ only has one prime ideal over $p$ and it is of degree $1$.
The same must hold for the suborder $\OO$.
Lemma~\ref{lem:O1O2} shows that the induced morphism $\overline{f_*} \colon \Chow(\OO') \to \Chow(\OO)$ is surjective.
Hence $\Chow(\OO)$ is trivial.
\end{proof}

\section{Further examples}

\begin{example}\label{ex:nonsplit}
The maximal order $\OK$ of $K:= \Q(\sqrt{-3}, \sqrt{13})$ has class number $2$.
Lemma~\ref{Order_aOO} shows that the order $\OO = \Z + 2\OK$ has a unique maximal ideal $\frakp$ over $2$.
Then $2\OK = \frakP \frakQ$ with maximal ideals $\frakP$ and $\frakQ$ which both generate the class group of $\OK$.
Since $\frakP$ and $\frakQ$ have degree $2$ over $\frakp$, the local Chow group $\Chow(\OO_\frakp)$ has order $2$.
Theorem~\ref{Thm:GlobalChow} shows 
\[ \im(\overline{f_*}) \isom \Cl(\OK) / \langle [\frakP/\frakQ] \rangle = \Cl(\OK) / \langle [\frakP^2] \rangle \isom \Cl(\OK) \isom \Z/2\Z \:. \]
So $ \Chow(\OO)$ is an extension of $\Z/2\Z$ by $\Z/2\Z$, but it is non-split, since Theorem~\ref{main} shows that $\Chow(\OO) \isom \Z/4\Z$ is generated by $\frakp = \frakP \cap \OO = \frakQ \cap \OO$.
\end{example}

\begin{example}\label{ex:NeitherNor}
The maximal order $\OK$ of
\[ K = \Q[X]/(X^5 - X^3 - X^2 + X + 1) \]
has trivial class group and
\[ 7 \OK = \frakP_1 \frakP_2 \]
splits into two maximal ideals with $[\OK: \frakP_1] = 7^2$ and $[\OK: \frakP_2] = 7^3$. 
The order
\[ \OO_1 := \Z + \frakP_1 \] 
has conductor $\frakP_1$ and index $7$ in $\OK$.
The ideal $\frakP_1$ is the unique non-invertible maximal ideal of $\OO_1$ and $\OO_1 / \frakP_1 \isom \F_{7}$.
Thus Theorem~\ref{Thm:GlobalChow} shows that 
\[ \Chow(\OO_1) \isom \Z/2\Z \]
is generated by $[\frakP_1]$.
Similarly, the order
\[ \OO_2 := \Z + \frakP_2 \]
has conductor $\frakP_2$ and index $7^2$ in $\OK$.
Again, $\frakP_2$ is the unique non-invertible maximal ideal of $\OO_2$ and $\OO_2 / \frakP_2 \isom \F_{7}$ implies that
\[ \Chow(\OO_2) \isom \Z/3\Z \]
is generated by $[\frakP_2]$.
The intersection $\OO':= \OO_1 \cap \OO_2$ has conductor $7\OK$ and index $7^3$ in $\OK$.
It has two different maximal ideals over $7$, namely $\frakp_i:= \OO' \cap \frakP_i$ for $i=1,2$.
Theorem~\ref{Thm:GlobalChow} shows that
\[ \Chow(\OO') \isom \Chow( \OO'_{\frakp_1} ) \times \Chow( \OO'_{\frakp_1} ) \isom \Z/2\Z \times \Z/3\Z \isom \Z/6\Z \]
is generated by $[\frakp_1]$ and $[\frakp_2]$.
Finally, the order $\OO:= \Z + 7 \OK$ also has conductor $7\OK$ but index $7^4$ in $\OK$.
By Lemma~\ref{Order_aOO} it only has a single maximal ideal $\frakp$ over $7$.
Theorem~\ref{Thm:GlobalChow} and the fact that $\OO/\frakp \isom \F_{7}$ show that $\Chow(\OO)$ must be trivial.\\
In particular, the Chow group homomorphism $\overline{f_*}$ from Proposition~\ref{Chow21} is in general neither injective nor surjective.
\end{example}

\section{A maximality condition}

Let $\OO$ be a one-dimensional noetherian domain with field of fractions $K$ and normalization $\OK$.
We assume that $\OK$ is finitely generated over $\OO$.
In \cite{MaximalI}, the author gives necessary and sufficient conditions when $\OO$ is normal in terms of the injectivity of some canonical homomorphisms.
Unfortunately, most results are wrong, see Example \ref{ex:inj}.
Since we already discussed the injectivity of some of these homomorphisms in Section~\ref{Sec:Chow}, we decided to correct the main results of \cite{MaximalI}.

%
As in Section~\ref{Sec:Chow}, let $\frakp_1,\dots,\frakp_r$ be the non-invertible maximal ideals of $\OO$ and let $\frakP_{i,1},\dots,\frakP_{i,r_i}$ be the maximal ideals of $\OK$ over $\frakp_i$.
Further let $\calF$ be the conductor of $\OO$ in $\OK$.

\begin{theorem}\label{prop:Fix}
The following statements are equivalent:
\begin{enumerate}
\item The canonical epimorphism $\calI(\OO) \to \calI(\OK)$ is injective.
\item The canonical epimorphism $\Pic(\OO) \to \Cl(\OK)$ is injective and $\OK^* = \OO^*$.
\item $(\OK/\calF)^* = (\OO/\calF)^*$.
\item $(\OK/\calF)^*$ is trivial.
\item The following conditions hold:
\begin{itemize}
\item $\calF$ is squarefree, i.e. $ \calF = \prod_{i=1}^r \prod_{j=1}^{r_i} \frakP_{i,j}$.
\item $\OK/\frakP_{i,j} \isom \F_2$ for all $1 \le i \le r$ and $1 \le j \le r_i$.
\item $r_i \ge 2$ for all $1 \le i \le r$.
\end{itemize}
\end{enumerate}
\end{theorem}
\begin{proof}
We first show that (1) implies (2). So suppose $\calI(\OO) \to \calI(\OK)$ is injective.
If $u \in \OK^*$, then the fractional principal ideal $u\OO$ is in the kernel of $\calI(\OO) \to \calI(\OK)$. Hence $u \in \OO^*$ and therefore $\OO^* = \OK^*$.
Suppose now $[\fraka]$ is in the kernel of $\Pic(\OO) \to \Cl(\OK)$.
Then $\fraka \OK = \alpha \OK$ for some $\alpha \in K^*$. Hence $\alpha^{-1} \fraka$ is in the kernel of $\calI(\OO) \to \calI(\OK)$.
This implies $\alpha^{-1} \fraka = \OO$. Hence $[\fraka] = [\OO]$.\\
To show that (2) implies (1), suppose $\Pic(\OO) \to \Cl(\OK)$ is injective and $\OO^* = \OK^*$.
If $\fraka$ is in the kernel of $\calI(\OO) \to \calI(\OK)$, then $[\fraka]$ is in the kernel of $\Pic(\OO) \to \Cl(\OK)$.
This shows that $\fraka = \alpha\OO$ for some $\alpha \in K^*$.
Then $\OK = \fraka \OK = \alpha \OK$ shows that $\alpha \in \OK^* = \OO^*$. Hence $\fraka = \OO$.\\
The equivalence of (2) and (3) is immediate from the exact sequence
\begin{equation}\label{eq:ES}
1 \to \OK^*/\OO^* \to (\OK/\calF)^*/(\OO/\calF)^* \to \Pic(\OO) \to \Cl(\OK) \to 1,
\end{equation}
see for example \cite[Proposition 12.9]{Neukirch}. 
The Chinese Remainder Theorem shows that (5) implies (4) which clearly implies (3).
So it remains to show that (3) implies (5).
For $x \in \OK$ and $\pi \in \prod_{i=1}^r \prod_{j=1}^{r_i} \frakP_{i,j}$ the ideal $(1 + \pi x) \OK$ is coprime to $\calF$ and thus 
\[ (1 + \pi x) + \calF \in (\OK/\calF)^* \overset{(2)}{=} (\OO/\calF)^* \:. \]
This implies $\pi x \in \OO$ and therefore $\pi \OK \subseteq \OO$.
So $\calF = \prod_{i=1}^r \prod_{j=1}^{r_i} \frakP_{i,j} $ is squarefree.
The inclusion $\frakp_i \subseteq \frakP_{i,j}$ shows that $\frakp_i \subseteq \bigcap_{j=1}^{r_i} \frakP_{i,j} = \prod_{j=1}^{r_i} \frakP_{i,j}$.
Conversely, $\prod_{j=1}^{r_i} \frakP_{i,j} \OO_{\frakp_i} = \calF \OO_{\frakp_i}$ is a proper ideal of the local ring $\OO_{\frakp_i}$ and thus $\frakp_{i} = \prod_{j=1}^{r_i} \frakP_{i,j}$.
Assertion (3) implies 
\[ 
(\OO_{\frakp_i}/\frakp_i \OO_{\frakp_i})^*
= (\OO_{\frakp_i}/\calF \OO_{\frakp_i})^* 
= (\OK_{\frakp_i}/\calF \OK_{\frakp_i})^* 
\isom \prod_{j=1}^{r_i} (\OK_{\frakP_{i,j}} / \frakP_{i,j})^*  \:.
\]
Hence $\deg_{\frakp_i}(\frakP_{i,j}) = 1$ for all $i,j$ and furthermore $r_i = 1$ or $\OO/\frakp_i \isom \F_2$.
If $r_i = 1$, then $\frakp_i = \frakP_{i,1}$ and thus $\OK/\frakp_i = \OK/\frakP_{i,1} \isom \OO/\frakp_i$ implies $\OK  = \OO + \frakp_i$.
Localizing at $\frakp_i$ and using Nakayama's Lemma shows $\OK_{\frakp_i} = \OO_{\frakp_i}$, which is impossible.
\end{proof}

\begin{example}\label{ex:inj}
Let $K = \Q(\sqrt{-7})$, $\OO = \Z[\sqrt{-7}]$ and $\OK = \Z[\frac{1+\sqrt{-7}}{2}]$ as in Example~\ref{ex:root7}.
It is well known that $\Cl(\OK)$ is trivial and clearly $\OO^* = \OK^* = \{\pm 1\}$.
The order $\OO$ has conductor $2\OK = \frakP \cdot \overline{\frakP}$ where $\frakP$ and $\overline{\frakP}$ are the two prime ideals of $\OK$ of norm~$2$.
Moreover, $2\OK$ is the unique prime ideal of $\OO$ over $2$. Thus
$(\OO/2\OK)^* \isom \F_2^*$ and $(\OK/2\OK) \isom (\OK/\frakP)^* \times (\OK/\overline{\frakP})^* \isom \F_2^* \times \F_2^*$ are both trivial.
Theorem~\ref{prop:Fix} shows that $\Pic(\OO) \isom \Pic(\OK)$ is also trivial and $\calI(\OO) \to \calI(\OK)$ is injective.
Finally, $\Chow(\OO)$ is trivial by Theorem~\ref{main}.
In particular, the maps $\calI(\OO) \to \calI(\OK)$, $\Pic(\OO) \to \Pic(\OK)$ and $\Pic(\OO) \to \Chow(\OO)$ are all injective.
This contradicts the wrong Theorems 2.1, 4.1 and 4.2 of \cite{MaximalI}.
\end{example}

Note that we can find further counterexamples whenever we have two different prime ideals of norm $2$ in $\OK$.


Theorem~\ref{Thm:divinj} below is \cite[Theorem 3.1]{MaximalI}. Since the proof presented there relies on the wrong statement \cite[Theorem 2.1]{MaximalI}, we decided to give a proof here.
We start by recalling a well known fact.

\begin{lemma}\label{lem:IIepi}
The canonical homomorphism
\[ \calI(\OO) \to \calI(\OK),\; \fraka \mapsto \fraka \OK \]
is surjective.
\end{lemma}
\begin{proof}
Let $\frakA \in \calI(\OK)$.
The surjectivity of $\Pic(\OO) \to \Cl(\OK)$ yields some $\fraka \in \calI(\OO)$ such that $[\frakA] = [\fraka \OK] \in \Cl(\OK)$.
Hence $\frakA = (\alpha \fraka) \OK$ for some $\alpha \in K^*$.
\end{proof}

\begin{theorem}\label{Thm:divinj}
The canonical homomorphism 
\[ \ddiv_\OO \colon \calI(\OO) \to \Div(\OO) \]
is injective if and only if $\OO = \OK$.
\end{theorem}
\begin{proof}
Suppose $\ddiv_\OO$ is injective.
Then $\OO^* = \OK^*$ by Corollary~\ref{count}.
Equation~\eqref{eq:fstar} shows that $\ddiv_\OO$ is the composition of the maps
\[ \calI(\OO) \longrightarrow \calI(\OK) \isom \Div(\OK) \overset{f_*}{\longrightarrow} \Div(\OK) \:. \]
Hence $\calI(\OO) \to \calI(\OK)$ must be injective and thus an isomorphism by Lemma~\ref{lem:IIepi}.
Thus $f_*$ must also be injective.
Lemma~\ref{lem:O1O2} shows that $r_i = 1$ for all $i$.
But this contradicts Theorem~\ref{prop:Fix} if $\OK \ne \OO$.
%
\end{proof}

\subsection*{Data Availability Statement} Data sharing not applicable to this article as no datasets were generated or analyzed during the current study.

\subsection*{Declarations} 
This work was funded by the Deutsche Forschungsgemeinschaft (DFG, German Research Foundation) -- Project-ID 491392403 -- TRR 358.

The authors have no financial or proprietary interests in any material discussed in this article.

\bibliography{chow}

\end{document}